\documentclass{birkmult}
\usepackage{hyperref, url}
\usepackage{soul,color}
 \newtheorem{theorem}{Theorem}[section]
 \newtheorem{lemma}[theorem]{Lemma}
 \numberwithin{equation}{section}

\begin{document}
%-------------------------------------------------------------------------
% editorial commands: to be inserted by the editorial office
%
%\firstpage{1}
%\volume{228}
%\Copyrightyear{2004}
%\DOI{003-0001}
%
%
%\seriesextra{Just an add-on}
%\seriesextraline{This is the Concrete Title of this Book\br H.E. R and S.T.C. W, Eds.}
%
% for journals:
%
%\firstpage{1}
%\issuenumber{1}
%\Volumeandyear{1 (2004)}
%\Copyrightyear{2004}
%\DOI{003-xxxx-y}
%\Signet
%\commby{inhouse}
%\submitted{March 14, 2003}
%\received{March 16, 2000}
%\revised{June 1, 2000}
%\accepted{July 22, 2000}
%
%
%
%---------------------------------------------------------------------------
%Insert here the title, affiliations and abstract:
%
\title[On an Interesting Class of Variable Exponents]
{On an Interesting Class of Variable Exponents}
%----------Author 1
\author[A.~Yu.~Karlovich]{Alexei Yu. Karlovich}
\address{%%
Departamento de Matem\'atica\\
Faculdade de Ci\^encias e Tecnologia\\
Universidade Nova de Lisboa\\
Quinta da Torre\\
2829--516 Caparica\\
Portugal}
\email{oyk@fct.unl.pt}

%----------Author 2
\author[I.~M.~Spitkovsky]{Ilya M. Spitkovsky}
\address{%%
Department of Mathematics\\
College of William \& Mary\\
Williamsburg, VA, 23187-8795\\
U.S.A.} \email{ilya@math.wm.edu}

%----------classification, keywords, date


%----------additions
\dedicatory{To Professor Stefan Samko on the occasion of
his 70th birthday}
%%% ----------------------------------------------------------------------
\begin{abstract}
Let $\mathcal{M}(\mathbb{R}^n)$ be the class of functions
$p:\mathbb{R}^n\to[1,\infty]$ bounded away from one and infinity
and such that the Hardy-Littlewood maximal function is bounded on
the variable Lebesgue space $L^{p(\cdot)}(\mathbb{R}^n)$. We
denote by $\mathcal{M}^*(\mathbb{R}^n)$ the class of variable
exponents $p\in\mathcal{M}(\mathbb{R}^n)$ for which
$1/p(x)=\theta/p_0+(1-\theta)/p_1(x)$ with some
$p_0\in(1,\infty)$, $\theta\in(0,1)$, and
$p_1\in\mathcal{M}(\mathbb{R}^n)$. Rabinovich and Samko
\cite{RS08} observed that each globally log-H\"older continuous
exponent belongs to $\mathcal{M}^*(\mathbb{R}^n)$. We show that
the class $\mathcal{M}^*(\mathbb{R}^n)$ contains many interesting
exponents beyond the class of globally log-H\"older continuous
exponents.
\end{abstract}
%%% ----------------------------------------------------------------------
\maketitle
%%% ----------------------------------------------------------------------
\section{Introduction}
Let $p:\mathbb{R}^n\to[1,\infty]$ be a measurable a.e. finite function. By
$L^{p(\cdot)}(\mathbb{R}^n)$ we denote the set of all complex-valued
functions $f$ on $\mathbb{R}^n$ such that
\[
I_{p(\cdot)}(f/\lambda):=\int_{\mathbb{R}^n} |f(x)/\lambda|^{p(x)} dx <\infty
\]
for some $\lambda>0$. This set becomes a Banach space when
equipped with the norm
\[
\|f\|_{p(\cdot)}:=\inf\big\{\lambda>0: I_{p(\cdot)}(f/\lambda)\le 1\big\}.
\]
It is easy to see that if $p$ is constant, then $L^{p(\cdot)}(\mathbb{R}^n)$ is nothing but
the standard Lebesgue space $L^p(\mathbb{R}^n)$. The space $L^{p(\cdot)}(\mathbb{R}^n)$
is referred to as a \textit{variable Lebesgue space}.
We will always suppose that
%%%
\begin{equation}\label{eq:exponents}
1<p_-:=\operatornamewithlimits{ess\,inf}_{x\in\mathbb{R}^n}p(x),
\quad
\operatornamewithlimits{ess\,sup}_{x\in\mathbb{R}^n}p(x)=:p_+<\infty.
\end{equation}
%%%
Under these conditions, the space $L^{p(\cdot)}(\mathbb{R}^n)$ is
separable and reflexive, and its dual is isomorphic to
$L^{p'(\cdot)}(\mathbb{R}^n)$, where
\[
1/p(x)+1/p'(x)=1 \quad(x\in\mathbb{R}^n)
\]
(see e.g.  \cite[Chap.~3]{DHHR11}).

Given $f\in L_{\rm loc}^1(\mathbb{R}^n)$, the Hardy-Littlewood maximal operator is
defined by
\[
Mf(x):=\sup_{Q\ni x}\frac{1}{|Q|}\int_Q|f(y)|dy
\]
where the supremum is taken over all cubes
$Q\subset\mathbb{R}^n$ containing $x$ (here, and throughout,
cubes will be assumed to have their sides parallel to the coordinate
axes). By $\mathcal{M}(\mathbb{R}^n)$ denote the set of all
measurable functions $p:\mathbb{R}^n\to[1,\infty]$ such that
\eqref{eq:exponents} holds and the Hardy-Littlewood maximal
operator is bounded on $L^{p(\cdot)}(\mathbb{R}^n)$.

Assume that \eqref{eq:exponents} is fulfilled. L.~Diening~\cite{D04}
proved that if $p$ satisfies
%%%
\begin{equation}\label{eq:log-Hoelder}
|p(x)-p(y)|\le\frac{c}{\log(e+1/|x-y|)} \quad (x,y\in\mathbb{R}^n)
\end{equation}
%%%
and $p$ is constant outside some ball, then
$p\in\mathcal{M}(\mathbb{R}^n)$. Further, the behavior of $p$ at
infinity was relaxed by D.~Cruz-Uribe, A.~Fiorenza, and
C.~Neugebauer \cite{CFN03,CFN04}, who showed that if $p$ satisfies
\eqref{eq:log-Hoelder} and there exists a $p_\infty>1$ such that
%%%
\begin{equation}\label{eq:Hoelder-infinity}
|p(x)-p_\infty|\le\frac{c}{\log(e+|x|)}\quad(x\in\mathbb{R}^n),
\end{equation}
%%%
then $p\in\mathcal{M}(\mathbb{R}^n)$. Following
\cite[Section~4.1]{DHHR11}, we will say that if conditions
\eqref{eq:log-Hoelder}--\eqref{eq:Hoelder-infinity} are fulfilled, then
$p$ is {\em globally log-H\"older continuous}.

A.~Nekvinda~\cite{N04,N08} relaxed condition
\eqref{eq:Hoelder-infinity}. To formulate his results, we will need the
notion of iterated logarithms. Put
\[
e_0:=1,\quad e_{k+1}:=\exp(e_k)\quad\mbox{for}\quad
k\in\mathbb{Z}_+:=\{0,1,2,\dots\}.
\]
The function $\log_k x$ is defined on the interval $(e_k,\infty)$
by
\[
\log_0 x:=x,\quad\log_{k+1}x:=\log(\log_k x)\quad\mbox{for}\quad
k\in\mathbb{Z}_+.
\]
For $\alpha>0$ and $k\in\mathbb{Z}_+$, put
\[
b_{k,\alpha}(x):=-\frac{1}{\alpha}\frac{d}{dx}(\log_k^{-\alpha}x)
\quad (x\ge e_k).
\]
We say that a measurable function $p:\mathbb{R}^n\to[1,\infty]$ belongs to
the Nekvinda class $\mathcal{N}(\mathbb{R}^n)$ if conditions
\eqref{eq:exponents}--\eqref{eq:log-Hoelder} are fulfilled and
there exists a monotone function $s:[0,\infty)\to[1,\infty)$
satisfying
%%%
\begin{equation}\label{eq:N1}
1<\inf_{x\in[0,\infty)}s(x), \quad
\sup_{x\in[0,\infty)}s(x)<\infty,
\end{equation}
%%%
and such that for some $K>0$, $k\in\mathbb{N}$, $\alpha>0$,
%%%
\begin{equation}\label{eq:N2}
\left|\frac{ds}{dx}(x)\right|\le
Kb_{k,\alpha}(x)\quad\mbox{for}\quad x\ge e_k,
\end{equation}
%%%
and
%%%
\begin{equation}\label{eq:N3}
\int_{\{x\in\mathbb{R}^n:p(x)\ne s(|x|)\}}c^{1/|p(x)-s(|x|)|}dx<\infty
\end{equation}
%%%
for some $c>0$. According to \cite[Theorem~2.2]{N08},
$\mathcal{N}(\mathbb{R}^n)\subset\mathcal{M}(\mathbb{R}^n)$. In
particular, all locally log-H\"older continuous (that is,
satisfying \eqref{eq:log-Hoelder}) radially monotone exponents
$p(x)=s(|x|)$ with monotone $s$ satisfying
\eqref{eq:N1}--\eqref{eq:N2} belong to
$\mathcal{M}(\mathbb{R}^n)$.

Observe, however, that A.~Lerner \cite{L05} (see also
\cite[Example~5.1.8]{DHHR11}) constructed exponents $p$
discontinuous at zero or at infinity and such that, nevertheless,
$p$ belong to $\mathcal{M}(\mathbb{R}^n)$. Thus neither
\eqref{eq:log-Hoelder} nor \eqref{eq:Hoelder-infinity} is
necessary for $p\in\mathcal{M}(\mathbb{R}^n)$. For more
informastion on the class $\mathcal{M}(\mathbb{R}^n)$ we refer to
\cite[Chaps.~4--5]{DHHR11}.

Finally, we note that V.~Kokilashvili and S.~Samko \cite{KS04,KS08};
V.~Kokilashvili, N.~Samko, and S.~Samko \cite{KSS06}; D.~Cruz-Uribe,
L.~Diening, and P.~H\"ast\"o \cite{CDH11} studied the boundedness
of the Hardy-Littlewood maximal operator on variable Lebesgue
spaces with weights under assumptions
\eqref{eq:exponents}--\eqref{eq:Hoelder-infinity} or their analogues in
the case of metric measure spaces.

We denote by $\mathcal{M}^*(\mathbb{R}^n)$ the collection of all variable exponents $p\in\mathcal{M}(\mathbb{R}^n)$
for which there exist constants $p_0\in(1,\infty)$, $\theta\in(0,1)$, and
a variable exponent $p_1\in\mathcal{M}(\mathbb{R}^n)$ such that
%%%
\begin{equation}\label{eq:definition}
\frac{1}{p(x)}=\frac{\theta}{p_0}+\frac{1-\theta}{p_1(x)}
\end{equation}
%%%
for almost all $x\in\mathbb{R}^n$.

This class implicitly appeared in V.~Rabinovich and S.~Samko's
paper~\cite{RS08} (see also \cite{S10}). Its introduction is
motivated by the fact that the boundedness of the Hardy-Littlewood
maximal operator on $L^{p_1(\cdot)}(\mathbb{R}^n)$ implies the
boundedness of many important linear operators on
$L^{p_1(\cdot)}(\mathbb{R}^n)$ (see e.g. \cite[Chap.~6]{DHHR11}).
If such a linear operator is also compact on the standard Lebesgue
space $L^{p_0}(\mathbb{R}^n)$, then, by a Krasnoselskii type
interpolation theorem for variable Lebesgue spaces, it is compact
on the variable Lebesgue space $L^{p(\cdot)}(\mathbb{R}^n)$ as
well.

In \cite[Theorem~5.1]{RS08}, the boundedness of the
pseudodifferential operators with symbols in the H\"ormander class
$S_{1,0}^0$ on the variable Lebesgue spaces
$L^{p(\cdot)}(\mathbb{R}^n)$ was established, provided that $p$
satisfies \eqref{eq:exponents}--\eqref{eq:Hoelder-infinity}. Then
the above interpolation argument was used in the proof of
\cite[Theorem~6.1]{RS08} to study the Fredholmness of
pseudodifferential operators with slowly oscillating symbols on
$L^{p(\cdot)}(\mathbb{R}^n)$. In particular, the following is
implicitly contained in the proof of \cite[Theorem~6.1]{RS08}.
%%%%%%%%%%%%%%%%%%%%%%%%%%%%%%%%%%%%%%%%%%%%%%%%%%%%%%%%%%%%%%%%%%%%%%%%%%%%%
\begin{theorem}[V.~Rabinovich-S.~Samko]
\label{th:RS}
If $p:\mathbb{R}^n\to[1,\infty]$ satisfies \eqref{eq:exponents}--\eqref{eq:Hoelder-infinity},
then $p$ belongs to $\mathcal{M}^*(\mathbb{R}^n)$.
\end{theorem}
%%%%%%%%%%%%%%%%%%%%%%%%%%%%%%%%%%%%%%%%%%%%%%%%%%%%%%%%%%%%%%%%%%%%%%%%%%%%%
Recently we generalized \cite[Theorem~5.1]{RS08} and proved that
the pseudodifferential operators with symbols in the H\"ormander
class $S_{\rho,\delta}^{n(\rho-1)}$, where $0\le\delta<1$ and
$0<\rho\le 1$, are bounded on variable Lebesgue spaces
$L^{p(\cdot)}(\mathbb{R}^n)$ whenever
$p\in\mathcal{M}(\mathbb{R}^n)$  (see \cite[Theorem~1.2]{KS11}).
Further, \cite[Theorem~1.3]{KS11} delivers a sufficient condition
for the Fredholmness of pseudodifferential operators with slowly
oscillating symbols in the H\"ormander class $S_{1,0}^0$ under the
assumption that $p\in\mathcal{M}^*(\mathbb{R}^n)$. The proof
follows the same lines as V.~Rabinovich and S.~Samko's proof of
\cite[Theorem~6.1]{RS08} for exponents satisfying
\eqref{eq:exponents}--\eqref{eq:Hoelder-infinity} and is based on
the above mentioned interpolation argument.

The aim of this paper is to show that the class
$\mathcal{M}^*(\mathbb{R}^n)$ is much larger than the class of
globally log-H\"older continuous exponents. Our first result says that
all Nekvinda's exponents belong to $\mathcal{M}^*(\mathbb{R}^n)$.
%%%%%%%%%%%%%%%%%%%%%%%%%%%%%%%%%%%%%%%%%%%%%%%%%%%%%%%%%%%%%%%%%%%%%%%%%%%%%
\begin{theorem}\label{th:Nekvinda}
We have
$\mathcal{N}(\mathbb{R}^n)\subset\mathcal{M}^*(\mathbb{R}^n)$.
\end{theorem}
%%%%%%%%%%%%%%%%%%%%%%%%%%%%%%%%%%%%%%%%%%%%%%%%%%%%%%%%%%%%%%%%%%%%%%%%%%%%%
Modifying A.~Lerner's example \cite{L05}, we further prove that there
are exponents in $\mathcal{M}^*(\mathbb{R}^n)$ that do not satisfy
\eqref{eq:Hoelder-infinity}.
%%%%%%%%%%%%%%%%%%%%%%%%%%%%%%%%%%%%%%%%%%%%%%%%%%%%%%%%%%%%%%%%%%%%%%%%%%%%%
\begin{theorem}\label{th:example}
There exists a sufficiently small $\varepsilon>0$ such that for
every $\alpha,\beta$ satisfying $0<\beta<\alpha\le\varepsilon$ the
function
\[
p(x)=2+\alpha+\beta\sin\big(\log(\log|x|)\chi_{\{x\in\mathbb{R}^n:|x|\ge e\}}(x)\big)
\quad (x\in\mathbb{R}^n)
\]
belongs to $\mathcal{M}^*(\mathbb{R}^n)$.
\end{theorem}
%%%%%%%%%%%%%%%%%%%%%%%%%%%%%%%%%%%%%%%%%%%%%%%%%%%%%%%%%%%%%%%%%%%%%%%%%%%%%
The paper is organized as follows. For completeness, we give a proof
of Theorem~\ref{th:RS} in Section~\ref{sec:RS}. Further, in
Section~\ref{sec:Nekvinda} we prove Theorem~\ref{th:Nekvinda}.
Section~\ref{sec:Lerner} contains A.~Lerner's sufficient condition for
$2+q\in\mathcal{M}(\mathbb{R}^n)$ in terms of mean oscillations of
a function $q$. In Section~\ref{sec:proof} we show that
Theorem~\ref{th:example} follows from the results of
Section~\ref{sec:Lerner}.
%%%%%%%%%%%%%%%%%%%%%%%%%%%%%%%%%%%%%%%%%%%%%%%%%%%%%%%%%%%%%%%%%%%%%%%%%%%%%
\section{Nekvinda's exponents}
\subsection{Globally log-H\"older continuous exponents are contained in \boldmath{$\mathcal{M}^*(\mathbb{R}^n)$}}
\label{sec:RS} In this subsection we give a proof of
Theorem~\ref{th:RS}. A part of this proof will be used in the
proof of Theorem~\ref{th:Nekvinda} in the next subsection.
%%%%%%%%%%%%%%%%%%%%%%%%%%%%%%%%%%%%%%%%%%%%%%%%%%%%%%%%%%%%%%%%%%%%%%%%%%%%%
\begin{proof}[Proof of Theorem~{\rm\ref{th:RS}}]
Suppose $p$ satisfies \eqref{eq:exponents}--\eqref{eq:Hoelder-infinity}. Let
$p_0\in(1,\infty)$, $\theta\in(0,1)$, and $p_1$ be such that \eqref{eq:definition}
holds. Then
%%%
\begin{equation}\label{eq:RS-1}
p_1(x)=\frac{p_0(1-\theta)p(x)}{p_0-\theta p(x)}.
\end{equation}
%%%
If we choose $p_0\ge p_+$, then for $x\in\mathbb{R}^n$,
%%%
\begin{equation}\label{eq:RS-2}
1<p_0-\theta p_+\le p_0-\theta p(x)\le p_0-\theta p_-<p_0.
\end{equation}
%%%
Therefore
\[
(1-\theta)p_-\le (1-\theta)p(x)\le p_1(x)\le p_0(1-\theta)p(x)\le p_0(1-\theta)p_+.
\]
Hence $(1-\theta)p_-\le (p_1)_-$ and $(p_1)_+\le p_0(1-\theta)p_+<\infty$. If we choose
$\theta$ such that $\theta\in(0,1-1/p_-)$, then $1<(p_1)_-$ and thus $p_1$ satisfies
\eqref{eq:exponents}.

From \eqref{eq:RS-1} it follows that
\[
p_1(x)-p_1(y)=\frac{p_0^2(1-\theta)(p(x)-p(y))}{(p_0-\theta p(x))(p_0-\theta p(y))}
\quad(x,y\in\mathbb{R}^n).
\]
Then, taking into account \eqref{eq:RS-2}, we get
%%%
\begin{equation}\label{eq:RS-3}
|p_1(x)-p_1(y)|\le p_0^2(1-\theta)|p(x)-p(y)|
\quad(x,y\in\mathbb{R}^n).
\end{equation}
%%%
Now put
\[
(p_1)_\infty:=\frac{p_0(1-\theta)p_\infty}{p_0-\theta p_\infty},
\]
where $p_\infty$ is the constant from \eqref{eq:Hoelder-infinity}.
Then
%%%
\begin{equation}\label{eq:RS-4}
p_1(x)-(p_1)_\infty=\frac{p_0^2(1-\theta)(p(x)-p_\infty)}{(p_0-\theta p(x))(p_0-\theta p_\infty)}
\quad (x\in\mathbb{R}^n).
\end{equation}
%%%
From \eqref{eq:Hoelder-infinity} it follows that
\[
p_\infty=\lim_{|x|\to\infty}p(x).
\]
Hence $p_\infty\in[p_-,p_+]$. Therefore
%%%
\begin{equation}\label{eq:RS-5}
1<p_0-\theta p_+\le p_0-\theta p_\infty\le p_0-\theta p_-<p_0.
\end{equation}
%%%
From \eqref{eq:RS-4} and \eqref{eq:RS-5} we obtain
%%%
\begin{equation}\label{eq:RS-6}
|p_1(x)-(p_1)_\infty|\le p_0^2(1-\theta)|p(x)-p_\infty|.
\end{equation}
%%%
From estimates \eqref{eq:RS-3}, \eqref{eq:RS-6} and
\eqref{eq:log-Hoelder}, \eqref{eq:Hoelder-infinity} for the
exponent $p$ we obtain that the exponent $p_1$ satisfies
\eqref{eq:log-Hoelder} and \eqref{eq:Hoelder-infinity}. Therefore
$p_1\in\mathcal{M}(\mathbb{R}^n)$ and thus $p$ belongs to $\mathcal{M}^*(\mathbb{R}^n)$.
\end{proof}
%%%%%%%%%%%%%%%%%%%%%%%%%%%%%%%%%%%%%%%%%%%%%%%%%%%%%%%%%%%%%%%%%%%%%%%%%%%%%
\subsection{Proof of Theorem~\ref{th:Nekvinda}}\label{sec:Nekvinda}
Suppose $p\in\mathcal{N}(\mathbb{R}^n)$. Let $p_0\in(1,\infty)$, $\theta\in(0,1)$,
and $p_1$ be such that \eqref{eq:definition} holds. In the previous subsection
we proved that if $p_0\ge p_+$ and $\theta\in(0,1-1/p_-)$, then $p_1$
satisfies \eqref{eq:exponents}--\eqref{eq:log-Hoelder}.

Since $p\in\mathcal{N}(\mathbb{R}^n)$, there exists a monotone function
$s:[0,\infty)\to[1,\infty)$ such that \eqref{eq:N1}--\eqref{eq:N3} are
fulfilled. Let
%%%
\begin{equation}\label{eq:Nekvinda-1}
s_1(x):=\frac{p_0(1-\theta)s(x)}{p_0-\theta s(x)}
\quad(x\ge 0).
\end{equation}
%%%
Put
%%%
\begin{equation}\label{eq:Nekvinda-2}
s_-:=\inf_{x\in[0,\infty)}s(x),
\quad
s_+:=\sup_{x\in[0,\infty)}s(x).
\end{equation}
%%%
We will choose $p_0$ and $\theta$ subject to
\[
p_0\ge\max\{p_+,s_+\},
\quad
\theta\in(0,\min\{1-1/p_-,1-1/s_-\}).
\]
Then, for $x\in[0,\infty)$,
%%%
\begin{equation}\label{eq:Nekvinda-3}
1<p_0-\theta s_+\le p_0-\theta s(x)\le p_0-\theta s_-<p_0.
\end{equation}
%%%
Therefore, for $x\in[0,\infty)$,
\[
(1-\theta)s_-\le (1-\theta) s(x)\le s_1(x)\le p_0(1-\theta) s(x)\le p_0(1-\theta)s_+
\]
and
%%%
\begin{equation}\label{eq:Nekvinda-4}
1<(1-\theta)s_-\le (s_1)_-,
\quad
(s_1)_+\le p_0(1-\theta)s_+<\infty,
\end{equation}
%%%
where $(s_1)_-$ and $(s_1)_+$ are defined by \eqref{eq:Nekvinda-2} with $s_1$
in place of $s$.

If $s(x)\le s(y)$ for $x,y\in[0,\infty)$, then
\[
p_0-\theta s(x)\ge p_0-\theta s(y),
\quad
p_0(1-\theta)s(x)\le p_0(1-\theta) s(y).
\]
Thus
\[
s_1(x)=\frac{p_0(1-\theta)s(x)}{p_0-\theta s(x)}
\le
\frac{p_0(1-\theta)s(y)}{p_0-\theta s(y)}=s_1(y),
\]
that is, $s_1$ is monotone.

It is easy to see that for almost all $x>0$,
\[
\frac{ds_1}{dx}(x)=\frac{p_0^2(1-\theta)}{(p_0-\theta s(x))^2}\,\frac{ds}{dx}(x).
\]
Taking into account \eqref{eq:Nekvinda-3}, we obtain
%%%
\begin{equation}\label{eq:Nekvinda-5}
\left|\frac{ds_1}{dx}(x)\right|\le p_0^2\left|\frac{ds}{dx}(x)\right|
\quad(x>0).
\end{equation}
%%%
From \eqref{eq:RS-1} and \eqref{eq:Nekvinda-1} we get
\[
E:=\{x\in\mathbb{R}^n:p(x)\ne s(|x|)\}=\{x\in\mathbb{R}^n:p_1(x)\ne s_1(|x|)\}
\]
and for $x\in E$,
\[
p_1(x)-s_1(|x|)=\frac{p_0^2(1-\theta)(p(x)-s(|x|))}{(p_0-\theta p(x))(p_0-\theta s(|x|))}.
\]
From this equality and inequalities \eqref{eq:RS-2} and \eqref{eq:Nekvinda-3}
we get for $x\in E$,
\[
(1-\theta)|p(x)-s(|x|)|\le |p_1(x)-s_1(|x|)|\le p_0^2(1-\theta)|p(x)-s(|x|)|.
\]
Therefore, there is a constant $M>0$ such that
%%%
\begin{equation}\label{eq:Nekvinda-6}
\int_E c^{1/|p_1(x)-s_1(|x|)|}\,dx\le M\int_E c^{1/|p(x)-s(|x|)|}\,dx.
\end{equation}
%%%
Since $s$ satisfies \eqref{eq:N1}--\eqref{eq:N3}, from \eqref{eq:Nekvinda-4}--\eqref{eq:Nekvinda-6}
it follows that $s_1$ satisfies \eqref{eq:N1}--\eqref{eq:N3}, too. Thus
$p_1\in\mathcal{N}(\mathbb{R}^n)$. By \cite[Theorem~2.2]{N08}, $p_1\in\mathcal{M}(\mathbb{R}^n)$,
which finishes the proof of $p\in\mathcal{M}^*(\mathbb{R}^n)$.
\qed
%%%%%%%%%%%%%%%%%%%%%%%%%%%%%%%%%%%%%%%%%%%%%%%%%%%%%%%%%%%%%%%%%%%%%%%%%%%%%
\section{Lerner's exponents}
\subsection{Sufficient condition for \boldmath{$2+q\in\mathcal{M}(\mathbb{R}^n)$}}\label{sec:Lerner}
Let $f\in L^1_{\rm loc}(\mathbb{R}^n)$.  For a cube $Q\subset\mathbb{R}^n$, put
\[
f_Q:=\frac{1}{|Q|}\int_Q f(x)dx.
\]
We recall that the mean oscillation of $f$
over a cube $Q$ is given by
\[
\Omega(f,Q):=\frac{1}{|Q|}\int_Q|f(x)-f_Q|dx.
\]
%%%%%%%%%%%%%%%%%%%%%%%%%%%%%%%%%%%%%%%%%%%%%%%%%%%%%%%%%%%%%%%%%%%%%%%%%%%%%%%
\begin{lemma}\label{le:Lipschitz-oscillation}
If $F\colon \mathbb{R}\to\mathbb{R}$ is a Lipschitz function with the
Lipschitz constant $c$ and $f\in L^1_{\rm loc}(\mathbb{R}^n)$ is a
real-valued function, then for every $Q\subset\mathbb{R}^n$,
\[
\Omega(F\circ f,Q)\le 2c\Omega(f,Q).
\]
\end{lemma}
%%%%%%%%%%%%%%%%%%%%%%%%%%%%%%%%%%%%%%%%%%%%%%%%%%%%%%%%%%%%%%%%%%%%%%%%%%%%%%
\begin{proof}
It is easy to see that
\[
\Omega(f,Q)\le\frac{1}{|Q|^2}\int_Q\int_Q|f(x)-f(y)|dx\, dy\le 2\Omega(f,Q).
\]
From this estimate we immediately get the statement.
\end{proof}
%%%%%%%%%%%%%%%%%%%%%%%%%%%%%%%%%%%%%%%%%%%%%%%%%%%%%%%%%%%%%%%%%%%%%%%%%%%%%%
Given any cube $Q$, let
\[
\ell(Q):=\log(e+\max\{|Q|,|Q|^{-1},|\operatorname{cen}_Q|\}),
\]
where $\operatorname{cen}_Q$ is the center of $Q$.
%%%%%%%%%%%%%%%%%%%%%%%%%%%%%%%%%%%%%%%%%%%%%%%%%%%%%%%%%%%%%%%%%%%%%%%%%%%%%%
\begin{lemma}[{see \cite[Proposition~4.2]{L05}}]
\label{le:double-logarithm}
If
\[
L(x):=\log(\log|x|)\chi_{\{x\in\mathbb{R}^n:|x|\ge e\}}(x)\quad (x\in\mathbb{R}^n),
\]
then
\[
\sup_{Q\subset\mathbb{R}^n}\ell(Q)\Omega(L,Q)<\infty.
\]
\end{lemma}
%%%%%%%%%%%%%%%%%%%%%%%%%%%%%%%%%%%%%%%%%%%%%%%%%%%%%%%%%%%%%%%%%%%%%%%%%%%%%
\begin{theorem}[{see \cite[Theorem~1.2]{L05}}]
\label{th:Lerner} There is a positive constant $\mu_n$, depending
only on $n$, such that for any measurable function
$q:\mathbb{R}^n\to\mathbb{R}$ with
\[
0<q_-:=\operatornamewithlimits{ess\,inf}_{x\in\mathbb{R}^n}q(x),
\quad
\|q\|_{L^\infty(\mathbb{R}^n)}+\sup_{Q\subset\mathbb{R}^n}\ell(Q)\Omega(q,Q)\le \mu_n,
\]
we have $2+q\in\mathcal{M}(\mathbb{R}^n)$.
\end{theorem}
%%%%%%%%%%%%%%%%%%%%%%%%%%%%%%%%%%%%%%%%%%%%%%%%%%%%%%%%%%%%%%%%%%%%%%%%%%%%%
\subsection{Proof of Theorem~\ref{th:example}}\label{sec:proof}
Let the function $L$ be as in Lemma~\ref{le:double-logarithm}.
Suppose $\alpha>\beta>0$ and put
\[
F(x):=\alpha+\beta\sin x\quad (x\in\mathbb{R})
\]
and
\[
q(y):=F(L(y)), \quad p(y):=2+q(y)\quad (y\in\mathbb{R}^n).
\]
Then
%%%
\begin{equation}\label{eq:example-1}
q_-=\alpha-\beta>0, \quad
\|q\|_{L^\infty(\mathbb{R}^n)}=\alpha+\beta<\infty.
\end{equation}
%%%
From Lemma~\ref{le:double-logarithm} we know that
%%%
\begin{equation}\label{eq:example-2}
C_L:=\sup_{Q\subset\mathbb{R}^n}\ell(Q)\Omega(L,Q)<\infty.
\end{equation}
%%%
Since $F$ is a Lipschitz function with the Lipschitz constant
equal to $\beta$, we obtain from
Lemma~\ref{le:Lipschitz-oscillation} that
%%%
\begin{align}
\sup_{Q\subset\mathbb{R}^n}\ell(Q)\Omega(q,Q) &=
\sup_{Q\subset\mathbb{R}^n}\ell(Q)\Omega(F\circ L,Q) \nonumber
\\
&\le 2\beta\sup_{Q\subset\mathbb{R}^n}\ell(Q)\Omega(L,Q)=2\beta C_L.
\label{eq:example-3}
\end{align}
%%%
From \eqref{eq:example-1}, \eqref{eq:example-3}, and $\beta<\alpha$ it follows that
%%%
\begin{equation}\label{eq:example-4}
\|q\|_{L^\infty(\mathbb{R}^n)}+\sup_{Q\subset\mathbb{R}^n}\ell(Q)\Omega(q,Q)
\le
\alpha+\beta+2\beta C_L
<
2\alpha(1+C_L).
\end{equation}

Fix $\theta\in(0,1)$ and take the function $G$ such that
\[
\frac{1}{2+F(x)}=\frac{\theta}{2}+\frac{1-\theta}{2+G(x)}\quad(x\in\mathbb{R}).
\]
Then
\[
\frac{1}{2+G(x)}
=
\frac{1}{1-\theta}\left(\frac{1}{2+F(x)}-\frac{\theta}{2}\right)
=
\frac{2-\theta(2+F(x))}{2(1-\theta)(2+F(x))}.
\]
Therefore
%%%
\begin{align*}
G(x)
&=
\frac{2(1-\theta)(2+F(x))}{2-\theta(2+F(x))}-2
\\
&=
\frac{2(1-\theta)(2+F(x))-4+2\theta(2+F(x))}{2-\theta(2+F(x))}
\\
&=
\frac{2(2+F(x))-4}{2-\theta(2+F(x))}
\\
&=
\frac{2F(x)}{2-\theta(2+F(x))}
\end{align*}
%%%
and
\[
G'(x)
=
\frac{2F'(x)(2-2\theta-\theta F(x))+2\theta F(x)F'(x)}{[2-\theta(2+F(x))]^2}
=
\frac{4(1-\theta)F'(x)}{[2-\theta(2+F(x))]^2}
\]
Since $\alpha>\beta>0$, we have $F(x)>0$ for $x\in\mathbb{R}$ and then
%%%
\begin{equation}\label{eq:example-5}
2-2\theta-\theta F(x)<2.
\end{equation}
%%%
Hence
\[
G(x)=\frac{2F(x)}{2-2\theta-\theta F(x)}>F(x).
\]
If we take $\theta=1/(2+\alpha+\beta)$, then
$\theta(2+F(x))\le\theta(2+\alpha+\beta)=1$. Therefore
$2-\theta(2+F(x))\ge 2-1=1$. Hence
%%%
\begin{equation}\label{eq:example-6}
G(x)=\frac{2F(x)}{2-2\theta-\theta F(x)}\le 2F(x)
\end{equation}
%%%
and
%%%
\begin{equation}\label{eq:example-7}
|G'(x)|
=
\frac{4(1-\theta)|F'(x)|}{[2-2\theta-\theta F(x)]^2}
\le
4(1-\theta)|F'(x)|<4|F'(x)|.
\end{equation}
%%%
Thus $G$ is Lipschitz with the Lipschitz constant equal to $4\beta$.

Put $q_1(y):=G(L(y))$ for $y\in\mathbb{R}^n$. Then from \eqref{eq:example-5}--\eqref{eq:example-6}
it follows that
%%%
\begin{equation}\label{eq:example-8}
(q_1)_-\ge\alpha-\beta>0,
\quad
\|q_1\|_{L^\infty(\mathbb{R}^n)}\le 2(\alpha+\beta).
\end{equation}
%%%
Further, from \eqref{eq:example-7} and Lemma~\ref{le:Lipschitz-oscillation}
we obtain
%%%
\begin{equation}\label{eq:example-9}
\sup_{Q\subset\mathbb{R}^n}\ell(Q)\Omega(q_1,Q)
=
\sup_{Q\subset\mathbb{R}^n}\ell(Q)\Omega(G\circ L,Q)\le 8\beta C_L.
\end{equation}
%%%
From \eqref{eq:example-8}--\eqref{eq:example-9} and $\beta<\alpha$
we deduce
\begin{equation}\label{eq:example-10}
\|q_1\|_{L^\infty(\mathbb{R}^n)}+\sup_{Q\subset\mathbb{R}^n}\ell(Q)\Omega(q_1,\Omega)
\le 2(\alpha+\beta)+8\beta C_L< 8\alpha (1+C_L).
\end{equation}
%%%
Let $\mu_n$ be the constant from Theorem~\ref{th:Lerner}. Put
\[
\varepsilon:=\frac{\mu_n}{8(1+C_L)}.
\]
If $0<\beta<\alpha\le\varepsilon$, then from \eqref{eq:example-1},
\eqref{eq:example-4} and \eqref{eq:example-8},
\eqref{eq:example-10} it follows that $q_->0$, $(q_1)_->0$, and
\[
\|q\|_{L^\infty(\mathbb{R}^n)}+\sup_{Q\subset\mathbb{R}^n}\ell(Q)\Omega(q,\Omega)\le\mu_n,
\quad
\|q_1\|_{L^\infty(\mathbb{R}^n)}+\sup_{Q\subset\mathbb{R}^n}\ell(Q)\Omega(q_1,\Omega)\le\mu_n.
\]
By Theorem~\ref{th:Lerner}, $p:=2+q\in\mathcal{M}(\mathbb{R}^n)$ and
$p_1:=2+q_1\in\mathcal{M}(\mathbb{R}^n)$. Hence,
\eqref{eq:definition}
holds with $\theta=1/(2+\alpha+\beta)$, $p,p_1\in\mathcal{M}(\mathbb{R}^n)$ and $p_0=2$.
Thus $p\in\mathcal{M}^*(\mathbb{R}^n)$.
\qed
%%%%%%%%%%%%%%%%%%%%%%%%%%%%%%%%%%%%%%%%%%%%%%%%%%%%%%%%%%%%%%%%%%%%%%%%%%%%%

\end{document}